\documentclass[a4paper, 11pt]{amsart}
\usepackage[
  margin=1.9cm,
  includefoot,
  footskip=30pt,
]{geometry}
\usepackage{amsmath,amssymb,amscd}
\usepackage{amsfonts}
\usepackage[english]{babel}
\usepackage[leqno]{amsmath}
\usepackage{amssymb,amsthm}
\usepackage{mathrsfs} 
\usepackage{amscd}
\usepackage{enumerate}
\usepackage{epsfig}
\usepackage{relsize}
\usepackage{layout}
\usepackage{fullpage}
\usepackage[usenames,dvipsnames]{xcolor}
\usepackage[backref=page]{hyperref}
\usepackage{tikz}
\hypersetup{
 colorlinks,
 citecolor=Green,
 linkcolor=Red,
 urlcolor=Blue}
\usepackage[matrix,arrow,tips,curve]{xy}
\input{xy}
\xyoption{all}
\definecolor{darkgreen}{rgb}{0.0, 0.7, 0.0}
\definecolor{purple}{rgb}{0.5, 0.0, 0.5}
\definecolor{red}{rgb}{0.8, 0.2, 0.0}

\newtheorem{thm}{Theorem}[section]

\newtheorem{lemma}[thm]{Lemma}

\newtheorem{prop}[thm]{Proposition}
\newtheorem{cor}[thm]{Corollary}
\newtheorem{claim}[thm]{Claim}

\numberwithin{equation}{section}
\theoremstyle{definition}
\newtheorem{defi}[thm]{Definition}

\newtheorem{question}[thm]{Question}

\newtheorem{stremark}[thm]{Remark}
\theoremstyle{remark}
\newtheorem{remark}[thm]{Remark}

\newcommand{\Q}{\mathbb{Q}}

\newcommand{\Pic}{\operatorname{Pic}}

\def \Im{{\rm Im}}

\def \PP{\mathbb{P}}

\def \GL{{\rm GL}}

\def \P{\mathcal P}

\def\I{\mathcal I}

\def \E{\mathcal E}

\def \H{\mathcal H}

\def\O{\mathcal O}

\def\M0{\mathcal M^0}

\def\mapright#1{\smash{\mathop{\longrightarrow}\limits^{#1}}}

\DeclareMathOperator{\codim}{{codim}}

\DeclareMathOperator{\Coker}{{Coker}}
\DeclareMathOperator{\Cliff}{{Cliff}}

\DeclareMathOperator{\Ker}{{Ker}}

\begin{document}

\title{%
On the extendability of projective varieties: a survey. \\ 
\MakeLowercase{(with an appendix by} T\MakeLowercase{homas} D\MakeLowercase{edieu)}}


\dedicatory{\normalsize  \ Dedicated to Ciro Ciliberto on the occasion of his 70th birthday. \hskip 6cm
What I learned from him, both from a mathematical and a human point of view, is invaluable.} 

\author[A.F. Lopez]{Angelo Felice Lopez}

\address{\hskip -.43cm Dipartimento di Matematica e Fisica, Universit\`a di Roma
Tre, Largo San Leonardo Murialdo 1, 00146, Roma, Italy. e-mail {\tt lopez@mat.uniroma3.it}}

\address{\hskip -.43cm Institut de Math\'ematiques de Toulouse~; UMR5219. Universit\'e de Toulouse~; CNRS.
UPS IMT, F-31062 Toulouse Cedex 9, France. \tt{thomas.dedieu@math.univ-toulouse.fr}.}

\thanks{* Research partially supported by  PRIN ``Advances in Moduli Theory and Birational Classification'' and GNSAGA-INdAM}

\thanks{{\it Mathematics Subject Classification} : Primary 14N05, 14J40. Secondary 14J28, 14H51.}

\begin{abstract} 
We give a survey of the incredibly beautiful amount of geometry involved with the problem of realizing a projective variety as hyperplane section of another variety.
\end{abstract}

\maketitle

\section{Introduction}

At the beginning of the '90s, I attended a seminar talk by Ciro Ciliberto at the University of Milan. The main upshot, beautifully conveyed by the speaker, was to make a connection between two apparently unrelated theorems that had appeared recently and to investigate the consequences of this awareness.

The story that followed in the subsequent years, and the research still going on today, will be recollected in this survey, through the magnifying lens of my views on the subject, very much influenced by that talk. 

To state the theorems in question, we start with some definitions and notation.

\begin{defi}
Let $X \subset \PP^r$ be an irreducible nondegenerate variety of codimension at least $1$. Let $k \ge 1$ be an integer. We say that $X$ is {\it $k$-extendable} if there exists a variety $Y \subset \PP^{r+k}$ different from a cone, with $\dim Y = \dim X +k$ and having $X$ as a section by an $r$-dimensional linear space. We say that $X$ is {\it precisely $k$-extendable} if it is $k$-extendable but not $(k+1)$-extendable. The variety $Y$ is called a {\it $k$-extension of $X$}. We say that $X$ is {\it extendable} if it is $1$-extendable.
\end{defi}

It is, of course, a basic question in projective geometry to understand when a variety is extendable and if so, how much.
Even when a variety is extendable, it can be extendable in different ways and with chains of extensions of different lengths (see the example in \cite{z}). 

On the other hand, unless one has a very good knowledge of the variety and its very ample line bundles, it is usually a difficult but fascinating problem. 

A general remark to be made is that two general approaches have been predominant. Researchers have often taken an "optimistic" or a "pessimistic" approach. The first one being: fix $X$ and try to prove that it is not extendable or to classify all of its extensions. The second one: start with $Y$, study its hyperplane sections and try to prove that such a $Y$ cannot exist. We will give examples of both.

To give a measure in this direction, let us define

\begin{defi}
\label{alfa}
Let $X \subset \PP^r$ be a smooth irreducible nondegenerate variety of codimension at least $1$ with normal bundle $N_X$. We set
$$\alpha(X) = h^0(N_X(-1))-r-1.$$
\end{defi}

The first result, due to F.L. Zak \cite{z} and S. L'vovsky \cite{lv} (see also \cite{ba2, bs}), is as follows.

\begin{thm} {\rm (Zak-L'vovsky's theorem)}
\label{zak}

Let $X \subset \PP^r$ be a smooth irreducible nondegenerate variety of codimension at least $1$ and suppose that $X$ is not a quadric. If $\alpha(X) \le 0$, then $X$ is not extendable. 

Given an integer $k \ge 2$, suppose that either:
\begin{itemize}
\item [(i)] $\alpha(X) < r$ (L'vovsky's version), or
\item [(ii)] $H^0(N_X(-2))=0$ (Zak's version).
\end{itemize} 
If $\alpha(X) \le k-1$, then $X$ is not $k$-extendable.
\end{thm}

Now to shift to the second result, let us first give the following

\begin{defi}
Let $C$ be a smooth irreducible curve. The {\it Wahl map of $C$} is the map 
$$\Phi_{\omega_C} : \Lambda^2 H^0(\omega_C) \to H^0(\omega_C^3)$$
defined by $\Phi_{\omega_C}(fdz \wedge gdz) = (fg'-gf') dz^3$.
\end{defi}

In the same years, J. Wahl \cite{w1} introduced this map (so named in the subsequent years) and proved the following seminal

\begin{thm} {\rm (Wahl's theorem)}
\label{wah}

Let $C$ be a smooth irreducible curve. If $C$ sits on a K3 surface, then $\Phi_{\omega_C}$ is not surjective.
\end{thm}

A beautiful geometric proof of this theorem was also given by Beauville and M\'erindol \cite{bm}.

The connection between Zak-L'vovsky's and Wahl's theorem will be made in the next sections.

\vskip .4cm

In this survey we will focus on the extendability of projective varieties. The question was investigated already by the Italian and British schools, it went on with the school of adjunction theory and received a big push after the theorems of Zak-L'vovsky and Wahl. It is a very active area of research still going strong today, where Ciro Ciliberto stands as one of the main contributors. 

We apologize for not treating, due to lack of space, several important aspects very much related to extendability, such as projective duality, deformation theory and graded pieces of the $T^1$ sheaf.

\section{Extendability in general}

In order to understand Zak-L'vovsky's theorem, we start with some simple but useful observations. Even though we do not know a precise reference, they have been all well-known for a long time. 

\begin{prop}
\label{inf}
Let $X \subset \PP^r = \PP V$ be a smooth irreducible nondegenerate variety of dimension $n \ge 1$ and of codimension at least $1$. Then:
\begin{itemize}
\item [(i)] $h^0(T_{{\PP^r}_{|X}}(-1)) \ge r+1$ and equality holds if $n \ge 2$;
\item [(ii)] $N_X(-1)$ is globally generated;
\item [(iii)] $\alpha(X) \ge 0$;
\item [(iv)] If $h^1(\O_X)=0$ and either $n \ge 3$ or $n=2$ and the multiplication map $V \otimes H^0(\omega_X) \to H^0(\omega_X(1))$ is surjective, then $\alpha(X) = h^1(T_X(-1))$.
\end{itemize} 
\end{prop}
\begin{proof}
From the twisted Euler sequence
\begin{equation}
\label{twe}
0 \to \O_X(-1) \to V^{\vee} \otimes \O_X \to T_{{\PP^r}_{|X}}(-1) \to 0
\end{equation}
we get that $T_{{\PP^r}_{|X}}(-1)$ is globally generated, $h^0(T_{{\PP^r}_{|X}}(-1)) \ge r+1$ and that equality holds if $n \ge 2$ by Kodaira vanishing. Hence we have (i). Now the twisted normal bundle sequence 
\begin{equation}
\label{norm2}
0 \to T_X(-1) \to T_{{\PP^r}_{|X}}(-1) \to N_X(-1) \to 0
\end{equation}
implies (ii). Moreover, since $H^0(T_X(-1)) = 0$ by \cite[Thm.~8]{ms}, \cite[Thm.~1]{w2} (unless $X$ is a plane conic, but then $\alpha(X) = 0$), we get that 
$$\alpha(X) = h^0(N_X(-1))-r-1 \ge h^0(T_{{\PP^r}_{|X}}(-1)) -r-1 \ge 0$$
that is (iii). Now assume that $h^1(\O_X)=0$. If $n \ge 3$, then $H^2(\O_X(-1))=0$ by Kodaira vanishing and we get that $H^1(T_{{\PP^r}_{|X}}(-1))=0$ by \eqref{twe}. When $n=2$ the same is achieved by observing that the map $H^2(\O_X(-1)) \to V^{\vee} \otimes H^2(\O_X)$ is injective. Then \eqref{norm2} gives (iv). 
\end{proof}

Let us take a look at the behaviour of $\alpha(X)$ under hyperplane sections.

\begin{prop}
\label{bd}
Let $X \subset \PP^r$ be a smooth irreducible nondegenerate variety of dimension $n \ge 1$ and of codimension at least $1$. Let $Y \subset \PP^{r+1}$ be a smooth extension of $X$. If $H^0(N_X(-2))=0$, then $H^0(N_Y(-2))=0$ and $\alpha(Y) \le \alpha(X)-1$.
\end{prop}
\begin{proof}
Since $(N_{Y/\PP^{r+1}})_{|X} \cong N_{X/\PP^r}$, for every $i \ge 1$ we have an exact sequence
\begin{equation}
\label{norm}
0 \to N_Y(-i-1) \to N_Y(-i) \to N_X(-i) \to 0.
\end{equation}
Since $H^0(N_X(-2))=0$ we get from \eqref{norm} that $h^0(N_Y(-i)) = h^0(N_Y(-i-1))$ for $i \ge 2$ and therefore $h^0(N_Y(-2))=0$. Then, again from \eqref{norm} with $i=1$, we see that
$$\alpha(Y) = h^0(N_Y(-1))-r-2 \le h^0(N_X(-1))-r-2 = \alpha(X)-1.$$
\end{proof}

We will now give an idea of how to prove a much-simplified version of Zak-L'vovsky's theorem (Zak's version). We will consider only the case of a chain of smooth extensions. For singular versions see \cite{bf, ba2}. In section \ref{can} we give an example of Zak of a variety that is extendable but not smoothly extendable.

\vskip .3cm

\noindent {\it Proof of Theorem \ref{zak}(simplified version).}
\begin{proof}
Recall that $\alpha(X) \ge 0$ by Proposition \ref{inf}(iii). We first record the following

\begin{claim}
\label{ext1}
If $\alpha(X) = 0$, then $H^0(N_X(-2))=0$.
\end{claim}
\begin{proof}
If $X$ has codimension $1$ and degree $d$, then $d \ge 3$ and $N_X(-1) \cong \O_X(d-1)$. But this gives
$$\alpha(X) \ge h^0(\O_X(2))-r-1>0$$
a contradiction. Therefore $\codim X \ge 2$. This implies that, at every point $x \in X$, the $\O_{X,x}$-module $N_X(-1)_x$ has rank at least $2$.

Assume that $H^0(N_X(-2))\ne 0$ and let $\sigma \in  H^0(N_X(-2))$ be non-zero. Let $\{\tau_0, \ldots, \tau_r\}$ be a basis of $\Im \{H^0(\O_{\PP^r}(1)) \to H^0(\O_X(1))\}$. Then $\sigma \otimes \tau_0, \ldots, \sigma \otimes \tau_r \in H^0(N_X(-1))$ are linearly independent, thus they are a basis. On the other hand, $N_X(-1)$ is globally generated by Proposition \ref{inf}(ii), hence the sections $\sigma \otimes \tau_j, 0 \le j \le r$ generate $N_X(-1)_x$. But the $\sigma_x \otimes (\tau_j)_x$ generate an $\O_{X,x}$-module of rank $1$, a contradiction.
\end{proof}

Therefore, if $\alpha(X) = 0$, the theorem follows by Claim \ref{ext1} and Propositions \ref{bd} and  \ref{inf}(iii). 

Next, let $k \ge 2$. Suppose that $H^0(N_X(-2))=0$ and that $X$ is $k$-extendable. We show, by induction on $k$, that $\alpha(X) \ge k$. This will of course complete the proof of the theorem.

Let $Y \subset \PP^{r+1}$ be a smooth extension of $X$. By Proposition \ref{bd} we have that 
$H^0(N_Y(-2))=0$ and $\alpha(X) \ge \alpha(Y)+1$. If $k=2$ we have that $Y$ is extendable, hence $ \alpha(Y) \ge 1$ by the first part of the theorem. Therefore $\alpha(X) \ge 2$. If $k \ge 3$, since $Y$ is $(k-1)$-extendable, we have by induction that $ \alpha(Y) \ge k-1$, hence $\alpha(X) \ge k$.
\end{proof}

\begin{remark} 
\label{nonmig}
It is easy to see that the conditions (i) and (ii) in Zak-L'vovsky's theorem are essential. For example one can take a hypersurface $X \subset \PP^{n+1}$ of degree $d \ge 3$. Then $\alpha(X) > n+1$ and $H^0(N_X(-2)) \ne 0$. On the other, hand $X$ is infinitely extendable. See Remark \ref{cansha} for an example of canonical curves.
\end{remark}

One immediate consequence of Zak-L'vovsky's theorem is the one mentioned below. As a matter of fact, this appeared before Theorem \ref{zak}, at least for smooth or normal extensions, and it actually has stronger consequences \cite{fuj, so, ba1, bs}, but we will not be concerned with them here.

\begin{prop} 
\label{tx}
Let $X \subset \PP^r$ be a smooth irreducible nondegenerate variety of dimension $n \ge 2$, of codimension at least $1$ and suppose that $X$ is not a quadric. If $H^1(T_X(-1)) = 0$, then $X$ is not extendable. 
\end{prop}
\begin{proof}
By \eqref{norm2} and \cite[Thm.~8]{ms}, \cite[Thm.~1]{w2} we see that $h^0(N_X(-1)) = h^0(T_{{\PP^r}_{|X}}(-1))$. Then $\alpha(X)=0$ by Proposition \ref{inf}(i), hence $X$ is not extendable by Theorem \ref{zak}.
\end{proof}

\section{How to estimate $\alpha(X)$: a fortunate coincidence}

Let $X \subset \PP^r$ be a smooth irreducible nondegenerate variety of dimension $n \ge 1$ and of codimension at least $1$.

How can we compute $\alpha(X)$? 

Unless one has a very good knowledge of the normal bundle $N_X$ (or of $T_X$ in the case of Proposition \ref{inf}(iv)), it turns out that it is quite difficult to compute $\alpha(X)$. 

However, in dimension $1$, a fortunate coincidence, explained below, comes to help. 

All the results that follow are of course not new (see for example \cite{cm2, w5}), but we include them for the benefit of the reader.  

\begin{lemma}
\label{cork}
Let $C \subset \PP^r$ be a smooth irreducible nondegenerate linearly normal curve of positive genus. Let
$\mu_{\O_C(1), \omega_C} : H^0(\O_C(1)) \otimes H^0(\omega_C) \to H^0(\omega_C(1))$
be the multiplication map on sections. 

Consider the map, called the Gaussian map,
$$\Phi_{\O_C(1), \omega_C}: H^0(\Omega_{{\PP^r}_{|C}} \otimes \omega_C(1)) \to H^0(\omega_{C}^2(1)).$$
Then:
\begin{itemize}
\item[(i)] $H^0(\Omega_{{\PP^r}_{|C}} \otimes \omega_C(1)) \cong \Ker \mu_{\O_C(1), \omega_C}$;
\item[(ii)] Given the identification in (i), we have that $\Phi_{\O_C(1), \omega_C}(s \otimes t) = sdt-tds$;
\item[(iii)] $\alpha(C) = {\rm cork} \Phi_{\O_C(1), \omega_C}$.
\end{itemize}
\end{lemma}
\begin{proof}

The twisted dual Euler sequence
$$0 \to \Omega_{{\PP^r}_{|C}} \otimes \omega_C(1) \to H^0(\O_C(1)) \otimes \omega_C \to \omega_C(1) \to 0$$
gives rise to the following exact cohomology sequence 
\begin{equation}
\label{eule}
0 \to H^0(\Omega_{{\PP^r}_{|C}} \otimes \omega_C(1)) \to 
H^0(\O_C(1)) \otimes H^0(\omega_C) \mapright{\mu_{\O_C(1), \omega_C}} H^0(\omega_C(1)) \to
\end{equation}
$$\hskip 7cm \to H^1(\Omega_{{\PP^r}_{|C}} \otimes \omega_C(1)) \to  H^0(\O_C(1)) \otimes H^1(\omega_C) \to 0.$$
This gives (i). Moreover, as $\mu_{\O_C(1), \omega_C}$ is surjective by \cite{c}, \cite[Thm.~1.6]{as}, we have from \eqref{eule} that
\begin{equation}
\label{h1}
h^1(\Omega_{{\PP^r}_{|C}} \otimes \omega_C(1)) = r+1.
\end{equation}
Now the twisted dual normal bundle sequence
$$0 \to N^{\vee}_{C} \otimes\omega_C(1) \to \Omega_{{\PP^r}_{|C}} \otimes \omega_C(1) \to \omega_{C}^2(1) \to 0$$
gives rise to the following exact cohomology sequence 
\begin{equation}
\label{cono}
H^0(\Omega_{{\PP^r}_{|C}} \otimes \omega_C(1)) \mapright{\Phi_{\O_C(1), \omega_C}}
H^0(\omega_{C}^2(1)) \to H^1(N^{\vee}_{C} \otimes \omega_C(1)) \to
H^1(\Omega_{{\PP^r}_{|C}} \otimes \omega_C(1)) \to 0.
\end{equation}
Thus we get (ii). Also, from \eqref{cono}, \eqref{h1} and Serre duality, we get
$$h^0(N_{C}(-1)) = r + 1 + {\rm cork} \Phi_{\O_C(1), \omega_C}$$ 
that is (iii).
\end{proof}
\begin{remark}
\label{mu1}
Both the Gaussian map $\Phi_{\O_C(1), \omega_C}$ and the Wahl map $\Phi_{\omega_C}$ have an ancestor in the map $\mu_1: \Ker  \mu_{W, \omega_C(-L)} \to H^0(\omega_C^2)$, where $W \subseteq H^0(L)$. This was introduced and studied in 1981 by Arbarello and Cornalba in \cite[\S 4]{ac} (see also \cite[Bib.~notes to Chapt.~XXI]{acg}, \cite[\S 4]{as}).
\end{remark}

By the above lemma, now the connection between Zak-L'vovsky's theorem and Wahl's theorem is clear from the following

\begin{prop}
\label{cano}
Let $C \subset \PP^{g-1}$ be a canonically embedded smooth curve of genus $g \ge 3$. Then
$$\alpha(C) = {\rm cork} \Phi_{\omega_C}.$$
Moreover, if $S \subset \PP^g$ is a smooth surface having $C$ as a hyperplane section, then $S$ is a K3 surface.
\end{prop}
\begin{proof}
First of all, Lemma \ref{cork}(iii) says that
$$\alpha(C) = {\rm cork} \Phi_{\omega_C, \omega_C}.$$
Now $\Phi_{\omega_C, \omega_C}$ vanishes on symmetric tensors by Lemma \ref{cork}(ii), hence
$$\alpha(C) = {\rm cork} \Phi_{\omega_C}.$$
For the second part, recall that $C$ is projectively normal by M. Noether's theorem. Hence the commutative diagram
$$\xymatrix{H^0(\O_{\PP^g}(l)) \ar[d] \ar@{->>}[r] & H^0(\O_{\PP^{g-1}}(l)) \ar@{->>}[d] \\ H^0(\O_S(l)) \ar[r] & H^0(\O_C(l))}$$
implies that the map $H^0(\O_S(l)) \to H^0(\O_C(l))$ is surjective for every $l \ge 0$. From the exact sequence
$$0 \to \O_S(l-1) \to \O_S(l) \to \O_C(l) \to 0$$
we deduce that $h^1(\O_S(l-1) \le h^1(\O_S(l))$ for every $l \ge 0$, and therefore $h^1(\O_S(l))=0$ for every $l \ge 0$.
In particular $q(S)=0$. Also $h^1(\O_S(1))=0$ and the exact sequence
$$0 \to \O_S \to \O_S(1) \to \omega_C \to 0$$
implies that $h^0(\omega_S) = h^2(\O_S) \ge h^1(\omega_C)=1$, hence $K_S \ge 0$. On the other hand 
$$\O_C(1) \cong \omega_C \cong \omega_S(1) \otimes \O_C$$
hence $K_S \cdot C = 0$ and this implies that $K_S=0$. Thus $S$ is a K3 surface.
\end{proof}

We remark that singular extensions of canonical curves can exist and are studied in \cite{e1, e2}.

We can now pair Proposition \ref{bd} and Lemma \ref{cork} to get the following way to estimate $\alpha(X)$.

\begin{prop}
Let $X \subset \PP^r$ be a smooth irreducible nondegenerate variety of dimension $n \ge 1$. Let $C$ be a smooth curve section of $X$ and assume that it is linearly normal and $H^0(N_C(-2))=0$. Then
$$\alpha(X) \le {\rm cork} \Phi_{\O_C(1), \omega_C}-n+1.$$
\end{prop}
\begin{proof}
Apply Lemma \ref{cork} and an iteration of Proposition \ref{bd}.
\end{proof}

The emerging philosophy is that if one can control the corank of the Gaussian maps of the curve section of $X$, then on can give information on the extendability of $X$. 

As far as we know, aside from passing to the curve section, there is only one other general result that allows studying the extendability of surfaces, without passing to the hyperplane section, but considering suitable linear systems on the surface. This will be discussed in section \ref{extsu}. 

However, in \cite{cdgk}, another interesting method for calculating $h^1(T_S(-1))$, when $S$ is an Enriques surface, is employed. This method potentially generalizes to other surfaces. 

Moreover, in \cite{r1, r2}, a very nice method of studying extendability of varieties covered by lines is introduced. Some applications of this method are mentioned in the next section.

\section{Old days}
\label{old}

Already in the beginning of the last century, several mathematicians of the Italian and British schools, such as Castelnuovo, Del Pezzo, Scorza, Terracini, Edge, Semple and Roth, just to mention a few, studied the extendability problem.

For example Scorza \cite{s1, s2, s3} proved that a Veronese variety of dimension $n \ge 2$ (see also \cite{seg} for $n=2$) or a Segre variety, except a quadric surface, is not extendable. Terracini \cite{te} also proved non-extendability of Veronese varieties. Di Fiore and Freni \cite{df} proved that a Grassmannian in its Pl\"ucker embedding, except $\mathbb G(1,3) \subset \PP^5$, is not extendable. 

We can quickly recover these results by applying Proposition \ref{tx}.

\begin{prop}
\label{h1t}
The following varieties $X \subset \PP^r$ of dimension $n \ge 2$ and codimension at least $1$, are not extendable:
\begin{itemize}
\item [(i)] any abelian variety or any finite unramified covering of such;
\item [(ii)] a Veronese variety;
\item [(iii)] a Segre variety, except a quadric surface;
\item [(iv)] a Grassmannian in its Pl\"ucker embedding, except $\mathbb G(1,3) \subset \PP^5$.
\end{itemize}
\end{prop}
\begin{proof}
For the Veronese surface $X \subset \PP^9$ observe that $H^1({\Omega^1_{\PP^9}}_{|X}) = \mathbb C$ by \eqref{twe}. Now the map 
$$\mathbb C = H^1(\Omega^1_{\PP^9}) \to H^1({\Omega^1_{\PP^9}}_{|X}) \to H^1(\Omega^1_X) = \mathbb C$$ 
is an isomorphism, hence so is
$H^1({\Omega^1_{\PP^9}}_{|X}) \to H^1(\Omega^1_X)$. Dualizing we get the isomorphism 
$$H^1(T_X(-1)) \to H^1({T_{\PP^9}}_{|X}(-1)).$$ 
Now using \eqref{norm2} and Proposition \ref{inf}(i) we get that $\alpha(X)=0$, hence Theorem \ref{zak} applies.

In all other cases, it is easily seen (for coverings use \cite[Prop.~2.7]{fuj}) that $H^1(T_X(-1))=0$, hence Proposition \ref{tx} applies.
\end{proof}

More generally and recently Russo, using a description of the Hilbert scheme of lines passing through a general point, proves in \cite{r1, r2} that irreducible hermitian symmetric spaces in their homogeneous embedding and adjoint homogeneous manifolds, with some obvious exceptions, are not extendable.

A celebrated result that also deserves to be mentioned in the recollection of the old days, is the so-called Babylonian tower theorem \cite{bv, bar, t, f, co}:

\begin{thm}

Let $X \subset \PP^r$ be a l.c.i. closed subscheme of pure codimension $c \ge 1$. If $X$ extends to a l.c.i. closed subscheme $Y  \subset \PP^{r+m}$ of pure codimension $c$ for all $m>0$, then $X$ is a complete intersection.
\end{thm}

On the other hand, if one drops the l.c.i. condition for all the extensions, then it is easy to get examples of smooth varieties that are infinitely extendable and are not complete intersections, such as arithmetically Cohen-Macaulay space curves (a more general result is in \cite{bal}).

We end this section by just mentioning many very nice results obtained by Beltrametti, Sommese, Van de Ven, Fujita, Badescu, Serrano, Lanteri, Palleschi, and in general by the people working in adjunction theory. It would be too long, for the purposes of this survey, to recall them here. We merely recall a short list of some relevant papers here \cite{sv, fuj1, se, fuj, so, lps1, lps2, lps3}, referring the reader to \cite{bs} and references therein.

\section{Extendability of canonical curves, K3 surfaces and Fano threefolds}
\label{can}

The most beautiful, in my view, part of the story is the incredible amount of nice mathematics that revolved around the study of extendability of canonical curves, still going strong today. This is very much intertwined with the study of extendability of K3 surfaces, as the second part of Proposition \ref{cano} shows.

Throughout this section, we will let $C \subset \PP^{g-1}$ be a canonically embedded smooth irreducible curve of genus $g \ge 3$ (just called a {\it canonical curve}).

While for $g = 3, 4$ we know that $C$ is a complete intersection in $\PP^{g-1}$, namely a plane quartic or a complete intersection of a quadric and a cubic in $\PP^3$, the first case to be considered, in terms of extendability, is $g=5$. In that genus the first phenomenon occurs: the curve might be trigonal or not. It is a famous theorem of Enriques and Petri that then $C$ is the intersection of three quadrics (hence extendable) if and only if it is not trigonal. Thus the Brill-Noether theory of $C$ comes into play, as far as extendability is concerned. As we will see, this is one of the main themes.

We now concentrate on what happens for a general curve, the meaning of general to be clear along the way.

In a series of beautiful papers, the cases of genus $6 \le g \le 11$ were settled by Mori and Mukai.

\begin{thm} \cite{m1, m2, m3, m4, mm}
\label{mukai}

Let $C \subset \PP^{g-1}$ be a canonical curve of genus $g$. Then:
\begin{itemize}
\item[(i)] If $g=6$, then $C$ extends to a quadric section of $G(2,5)$ if and only if $C$ has finitely many $g^1_4$'s;
\item[(ii)] If $g=7$, then $C$ extends to a section of $OG(5,10)$ if and only if $C$ has no $g^1_4$'s;
\item[(iii)] If $g=8$, then $C$ extends to a section of $G(2,6)$ if and only if $C$ has no $g^2_7$'s;
\item[(iv)] If $g=9$, then $C$ extends to a section of $SpG(3,6)$ if and only if $C$ has no $g^1_5$'s;
\item[(v)] If $g=10$, then a general $C$ is not smoothly extendable;
\item[(vi)] If $g=11$, then a general $C$ is smoothly extendable.
\end{itemize}
\end{thm}

In the above theorem, $G(k,n)$ is the usual Grassmannian, $OG(k,2n)$ is the orthogonal Grassmannian, that is the variety that parametrizes $k$-dimensional vector subspaces of $\mathbb C^{2n}$ that are isotropic with respect to a non-degenerate symmetric bilinear form on $\mathbb C^{2n}$ and $SpG(n,2n)$ is the symplectic Grassmannian, that is, the Grassmannian of Lagrangian subspaces of a $2n$-dimensional symplectic vector space.

Taking into account the second part of Proposition \ref{cano}, the above theorem clearly suggests that one should look more closely at the moduli spaces of curves and K3 surfaces.

Let $\mathcal M_g$ be the moduli space of curves of genus $g$. Let $\mathcal K_g$ be the moduli stack of $K3$ surfaces  of genus $g$, that is pairs $(S, L)$ with $S$ a smooth K3 surface, and $L$ an ample, globally generated line bundle on $S$ with $L^2 = 2g - 2$. Let $\mathcal K \mathcal C_g$ be the moduli space of pairs $(S,C)$ such that $(S, \O_S(C)) \in \mathcal K_g$. Then one has a forgetful morphism
\begin{equation}
\label{pi}
c_g : \mathcal K \mathcal C_g \to \mathcal M_g
\end{equation}
and clearly the question is when is $c_g$ dominant. Since $\dim \mathcal K \mathcal C_g = g+19$ and $\dim \mathcal M_g = 3g-3$ one gets that this is possible if $g \le 11$. Now, on the one hand, Theorem \ref{mukai} shows that $c_g$ is dominant if $g \le 9$ or $g=11$. On the other hand, a clean connection with Wahl's theorem, was made by Ciliberto, Harris and Miranda (later proved also by Voisin).

\begin{thm} \cite{chm, v}
\label{chm}

Let $C$ be a general curve of genus $g = 10$ or $g \ge 12$. Then $\Phi_{\omega_C}$ is surjective. In particular $C$, in its canonical embedding, is not extendable.
\end{thm}
The second part (see also \cite[Cor.~page 26]{hm}) follows by Proposition \ref{cano} and Zak-L'vovsky's theorem.

Once the dominance of $c_g$ is settled, the next question in order is the study of its nonempty fibers, or, in other words, the study of how many ways can a hyperplane section of a K3 surface extend. A speculation about this was already made in \cite{cm}, where some evidence was presented to the idea that the corank of the Wahl map, in that case, for $g=11$ or $g \ge 13$, should be one. This speculation turned out to be true.

We will use the following.

\begin{defi}
\label{pri}
A {\it prime K3 surface of genus $g$} is a smooth K3 surface $S \subset \PP^g$ of degree $2g-2$ such that $\Pic(S)$ is generated by the hyperplane bundle. We denote by $\H_g$ the unique component of the Hilbert scheme of prime K3 surfaces of genus $g$.
\end{defi}

The idea of \cite{clm1} is to degenerate a prime K3 surface to the union of two rational normal scrolls and then to degenerate the latter to a union of planes whose hyperplane section is a suitable graph curve having corank one Wahl map. The result obtained is the following

\begin{thm} {\rm (corank one theorem \cite[Thm.~4]{clm1})}
\label{ck1}

Suppose that $g=11$ or $g \ge 13$. Let $S$ be a K3 surface represented by a general point in $\H_g$ and let $C$ be a general hyperplane section of $S$. Then ${\rm cork} \Phi_{\omega_C} = 1$.
\end{thm}

As a matter of fact, when $g \ge 17$, the same result holds for any smooth section of $S$ (see \cite[Thm.~2.15 and Rmk.~2.23]{clm2}).

Together with some other computations of coranks (see \cite{clm1}), there are two immediate consequences. 

The first one is about the map $c_g$ (see \eqref{pi}).

\begin{thm} \cite[Thm.~5]{clm1}

\begin{itemize}
\item[(i)] If $3 \le g \le 9$ and $g=11$, then $c_g$ is dominant and the general fiber is irreducible;
\item[(ii)] $\codim \Im c_{10} = 1$;
\item[(iii)] $\codim \Im c_{12} = 2$;
\item[(iv)] If $g=11$ or $g \ge 13$, then $c_g$ is birational onto its image;
\item[(v)] If $g=11$ or $g \ge 13$, then a general canonical curve that is a hyperplane section of a prime K3 surface, lies on a unique one, up to projective transformations.
\end{itemize}
\end{thm}

The fact that $c_g$ is birational onto its image was later reproved by Mukai \cite{m7, m8}, Arbarello, Bruno and Sernesi \cite{abs1} and  Feyzbakhsh \cite{fe1, fe2, fe3}. This is part of Mukai's beautiful program of reconstructing the K3 surface from a moduli space of sheaves on the hyperplane section. We will not be concerned about these matters here.
 
Returning now to the degeneration method explained after Definition \ref{pri}, it allows also to prove that $H^0(N_C(-2))=0$ for a general canonical curve $C$ of genus $g \ge 7$ that is hyperplane section of a prime K3 surface \cite[Lemma 4]{clm1}. Together with the ``corank one theorem" (Theorem \ref{ck1}), this gives rise to a simple application of Zak-L'vovsky's theorem and Proposition \ref{cano}.

We will use the following.

\begin{defi}
A {\it prime Fano threefold of genus $g$} is a smooth anticanonically embedded threefold $X \subset \PP^{g+1}$ of degree $2g-2$ such that $\Pic(X)$ is generated by the hyperplane bundle. We denote by $\mathcal V_g$ the Hilbert scheme of prime Fano threefolds of genus $g$.
\end{defi}

Then we have the following results, reproving in a quite simple way the important classification results of Fano threefolds \cite{i1, i2, m5, m6}.

\begin{thm} \cite[Thm.~6 and 7]{clm1}

\begin{itemize}
\item[(i)] If $g=11$ or $g \ge 13$, then there is no prime Fano threefold of genus $g$;
\item[(ii)] If $6 \le g \le 10$ or $g=12$, then $\mathcal V_g$ is irreducible and the examples of Fano and Iskovskih fill out all of $\mathcal V_g$.
\end{itemize}
\end{thm}

Concerning the last statement, the meaning is that just the knowledge of an example for every $g$ allows concluding that they fill out the Hilbert scheme. The main idea behind (ii) (and also behind Theorem \ref{mukvar} below) being that a prime Fano threefold is projectively Cohen-Macaulay, hence (see \cite[page 46]{p}) flatly degenerates to the cone over its hyperplane section $S$. Now knowledge of the cohomology of the normal bundle of $S$ allows proving that this cone is a smooth point of the Hilbert scheme. Then to obtain irreducibility one uses the fact that such K3 surfaces $S$ are general in $\H_g$.

Another important observation in \cite[Table 2]{clm1} is that, if $6 \le g \le 10$ or $g=12$, and $C$ is a general hyperplane section of a general $K3$ surface of genus $g$, then ${\rm cork} \Phi_{\omega_C} > 1$, thus suggesting that the curve might be precisely (${\rm cork} \Phi_{\omega_C}$)-extendable. 

Also this turned out to be true and related to another milestone in the study of Fano varieties: Fano threefolds of index greater than $1$ and Mukai varieties.

We will state, for simplicity of exposition, only the results about Picard rank $1$ and $g \ge 3$. The interested reader can consult \cite{clm2, cd1} for further results.

\begin{defi}
For $r \ge 2$, let $\H_{r,g}$ be the component of the Hilbert scheme whose general elements are obtained by embedding prime K3 surfaces of genus $g$ via the $r$-th multiple of the primitive class.

We denote by $\mathcal V_{n, r, g}$ the Hilbert scheme of smooth nondegenerate varieties $X \subset \PP^N$ of dimension $n \ge 3$, such that $\rho(X)=1$ and whose general surface section is a K3 surface represented by a point in $\H_{r,g}$. For $n \ge 4$, a {\it Mukai variety} is a variety $X$ with $\rho(X)=1$ represented by a point in $\mathcal V_{n, 1, g}$.
\end{defi}

With similar computations of coranks and cohomology of the normal bundle, the following classification results were obtained in \cite{clm2} (note also \cite{clm2c}, even though the classification results are not affected).

The first one is for Fano threefolds of index $r \ge 2$. Again, as the classification was known before, the novelty is (ii).

\begin{thm} \cite[Thm.~3.2]{clm2}
\label{mukvar}

\begin{itemize}
\item[(i)] If $(r,g) \not\in\{(2,3), (2,4), (2,5), (2,6), (3,4), (4,3)\}$, then there is no Fano threefold in $\mathcal V_{3, r, g}$;
\item[(ii)] If $(r,g) \in\{(2,3), (2,4), (2,5), (2,6), (3,4), (4,3)\}$, then $\mathcal V_{3, r, g}$ is irreducible and the examples of Fano and Iskovskih form a dense open subset of smooth points of $\mathcal V_{3, r, g}$.
\end{itemize}
\end{thm}

Similarly, for Mukai varieties, classified by Mukai himself \cite{m6}, we have (for the definition of $n(g)$ see \cite[Table 3.14]{clm2}) the following, where again (ii) was new.

\begin{thm} \cite[Thm.~3.15]{clm2}
\label{clm1}

\begin{itemize}
\item[(i)] If $(r,g) \not\in\{(1,s), 6 \le s \le 10, (2,5)\}$ or if $(r,g) \in \{(1,s), 6 \le s \le 10, (2,5)\}$ and $n > n(g)$, then there is no Mukai variety in $\mathcal V_{n, 1, g}$;
\item[(ii)] If $(r,g) \in\{(1,s), 6 \le s \le 10, (2,5)\}$ and $4 \le n \le n(g)$, then $\mathcal V_{n, 1, g}$ is irreducible and the examples of Mukai form a dense open subset of smooth points of $\mathcal V_{n, 1, g}$.
\end{itemize}
\end{thm}

\vskip .4cm

So much for general canonical curves! What about extendability of {\it any} canonical curve?

\vskip .4cm

One way to measure how far a curve is from being general is via its Clifford index, that we now recall.

\begin{defi}
Let $C$ be a smooth irreducible curve of genus $g \ge 4$. The Clifford index of $C$ is
$$\Cliff(C) = \min\{\deg L - 2 r(L), L \ \hbox{a line bundle on} \ C \ \hbox{such that} \ h^0(L) \ge 2, h^1(L) \ge 2\}.$$
\end{defi}

Recall that $\Cliff(C) = 0$ if and only if $C$ is hyperelliptic; $\Cliff(C) = 1$ if and only if  $C$ is trigonal or isomorphic to a plane quintic; $\Cliff(C) = 2$ if and only if $C$ is tetragonal or isomorphic to a plane sextic. A general curve of genus $g$ has Clifford index $\lfloor \frac{g-1}{2} \rfloor$.

Already Beauville and M\'erindol, in their proof of Wahl's theorem, observed that if a smooth irreducible curve $C$ sits on a K3 surface $S$, then the surjectivity of $\Phi_{\omega_C}$ implies the splitting of the normal bundle sequence
$$0 \to T_C \to {T_S}_{|C} \to N_{C/S} \to 0.$$
This introduced the idea, exploited by Voisin in \cite{v}, that the elements of $(\Coker \Phi_{\omega_C})^{\vee}$ should be interpreted as ribbons, or infinitesimal surfaces, embedded in $\PP^g$ and extending a canonical curve $C$. We will return to this later.

In 1996, Wahl proposed a possible converse to his theorem (Theorem \ref{wah}), that, as we will see, is very much connected with Voisin's idea and further developments.

\begin{thm} \cite[Thm.~7.1]{w4}
\label{sq}

Let $C \subset \PP^{g-1}$ be a canonical curve of genus $g \ge 8$ and $\Cliff(C) \ge 3$. Suppose that $C$ satisfies
\begin{equation}
\label{sq1}
H^1(\I^2_{C/\PP^{g-1}}(k))=0 \ \hbox{for every} \ k \ge 3.
\end{equation}
Then $C$ is extendable if and only if $\Phi_{\omega_C}$ is not surjective.
\end{thm}

Concerning the condition \eqref{sq1}, Wahl proved that it is satisfied by a general canonical curve, while it is not satisfied by a general tetragonal curve of genus $g \ge 8$. He also conjectured that \eqref{sq1} holds for every canonical curve $C$ with $\Cliff(C) \ge 3$. This conjecture was then proved, in almost all cases, in a beautiful paper by Arbarello, Bruno and Sernesi.

\begin{thm} \cite[Thm.~1.3]{abs}
\label{abs}

Let $C \subset \PP^{g-1}$ be a canonical curve of genus $g \ge 11$ and $\Cliff(C) \ge 3$. Then 
$$H^1(\I^2_{C/\PP^{g-1}}(k))=0 \ \hbox{for every} \ k \ge 3.$$
\end{thm}

Besides playing a crucial role in what follows, this theorem, together with Wahl's theorem and Theorem \ref{sq}, gives

\begin{cor} \cite[Cor.~1.4]{abs}
\label{abs2}

Let $C \subset \PP^{g-1}$ be a canonical curve of genus $g \ge 11$ and $\Cliff(C) \ge 3$. Then $C$ is extendable if and only if $\Phi_{\omega_C}$ is not surjective.
\end{cor}

In the same paper Arbarello, Bruno and Sernesi, also proved another important result. We first recall

\begin{defi}
A smooth irreducible curve $C$ is called a {\it Brill–Noether–Petri curve} if the multiplication map $H^0(L) \otimes H^0(K_C-L) \to H^0(K_C)$ is injective for every line bundle $L$ on $C$. 
\end{defi}

Following an idea of Mukai, Voisin suggested in \cite{v} to study the relation between the extendability of Brill–Noether–Petri curves to K3 surfaces and the non-surjectivity of $\Phi_{\omega_C}$. Wahl conjectured in \cite{w4} that for $g \ge 8$ this is an equivalence. While this conjecture turned out to be false \cite{abfs, ab}, it was essentially true and this is the other result by Arbarello, Bruno and Sernesi mentioned above. 

\begin{thm} \cite[Thm.~1.1]{abs}
\label{bn}

Let $C$ be a Brill–Noether–Petri curve of genus $g \ge 12$. Then $C$ lies on a polarized K3 surface, or on a limit thereof, if and only if its Wahl map is not surjective.
\end{thm}

This circle of ideas and results was closed, as of now, by several beautiful theorems proved in \cite{cds}. The paper is full of very interesting results, we will describe only some of them.

The first one is that, conversely to Zak-L'vovsky's theorem, the condition $\alpha(X) \ge k$ is sufficient for the $k$-extendability of most canonical curves and K3 surfaces. This was obtained by a generalization of the method of Voisin and Wahl of integrating ribbons, as follows. Note that the vanishing in Theorem \ref{abs} is fundamental for this method to work.

\begin{thm} \cite[Thm.~2.1]{cds}
\label{cds1}

Let $C$ be a smooth curve of genus $g$ with $\Cliff(C) \ge 3$. Consider the following two statements:
\begin{itemize}
\item[(i)] ${\rm cork} \Phi_{\omega_C} \ge k+1$;
\item[(ii)] There exists an arithmetically Gorenstein normal variety $Y \subset \PP^{g+k}$, not a cone, with $\dim(Y) = k + 2, \omega_Y = \O_Y (-k)$, which has a canonical image of $C$ as a section with a $(g-1)$-dimensional linear subspace of $\PP^{g+k}$ (in particular, $C \subset \PP^{g-1}$ is $(k + 1)$-extendable).
\end{itemize}
If $g \ge11$, then (i) implies (ii). Conversely, if $g \ge 22$ and the canonical image of $C$ is a hyperplane section of a smooth K3 surface in $\PP^g$, then (ii) implies (i).
\end{thm}

As a matter of fact, the extension $Y$ in the theorem above is universal, in the sense that it gives all surface extensions of $C$, see \cite[Cor.~5.5]{cds}. A general criterion for universal extensions of projective varieties was given in \cite[Thm.~A]{w6}. The latter paper contains also nice results about extending a variety to a Calabi-Yau variety.

As for the corank of the Wahl map, the following result is proved by Ciliberto, Dedieu and Sernesi.

\begin{thm} \cite[Cor.~2.10]{cds}
\label{cds4}

Let $C$ be a smooth curve of genus $g > 37$ with $\Cliff(C) \ge 3$. If the canonical model of $C$ is a hyperplane section of a K3 surface $S$, possibly with ADE singularities, then ${\rm cork} \Phi_{\omega_C} = 1$.
\end{thm}

It is an intriguing open question to find examples of higher corank.

\begin{question} \cite[Question 2.14]{cds}

Does there exist any Brill-Noether general curve of genus $g \ge 11, g \ne 12$, such that ${\rm cork} \Phi_{\omega_C} > 1$?
\end{question} 

As explained in  \cite[Rmk.~2.13]{cds}, all canonical curves in smooth Fano threefolds with Picard number greater than $1$ are Brill-Noether special.

Before stating the next result, let us define

\begin{defi}
The Clifford index $\Cliff(S,L)$ of a polarized K3 surface $(S,L)$ is the Clifford index of any smooth curve $C \in |L|$.
\end{defi}

Note that this does not depend on the choice of $C$ by \cite{gl}. In the definition is assumed that a general curve in $|L|$ is smooth and irreducible, which, being $L$ ample, happens by \cite[Prop. 8.1 and Thm. 3.1]{sd} unless $L = kE+R$ with $k \geq 3, E^2=0, R^2=-2, E.R=1$.

Now if $C$ is a smooth hyperplane section of a K3 surface $S \subset \PP^g$, then $\Cliff(C) = \Cliff(S,L)$, and in \cite[Cor.~2.8]{cds}, as a consequence of the proof of Theorem \ref{cds3} below, the authors prove that  
$${\rm cork} \Phi_{\omega_C} = h^1(T_S(-1)) + 1.$$ 

Then Theorem \ref{cds1} gives

\begin{thm} \cite[Thm.~2.18]{cds}
\label{cds2}

Let $(S, L) \in \mathcal K_g$ be a polarized K3 surface with $\Cliff(S, L) \ge 3$. Consider the following two statements:
\begin{itemize}
\item[(i)] $h^1(T_S \otimes L^{\vee}) \ge k$;
\item[(ii)] There exists an arithmetically Gorenstein normal variety $X \subset \PP^{g+k}$, not a cone, with $\dim(X) = k + 2, \omega_X = \O_X(-k)$, having the image of $S$ by the linear system $|L|$ as a section with a $g$-dimensional linear subspace of $\PP^{g+k}$.
\end{itemize}
If $g \ge11$, then (i) implies (ii). Conversely, if $g \ge 22$, then (ii) implies (i).
\end{thm}

On the other hand, extendability of K3 surfaces, is possible only for bounded genus, using results of Prokhorov \cite{p1} and Karzhemanov \cite{k, k2}:

\begin{thm} \cite[2.9]{cds}
\label{cds5}

Let $S \subset \PP^g$ be a K3 surface possibly with ADE singularities. If $g = 35$ or $g \ge 38$, then $S$ is not extendable.
\end{thm}

This bound is sharp \cite{z, p1}. Moreover, the example of genus $37$ is a K3 surface that is extendable but not smoothly extendable.

As for the map $c_g$ (see \eqref{pi}), Ciliberto, Dedieu and Sernesi prove that the fibers are smooth over curves with Clifford index at least $3$ and $g \ge 11$.

\begin{thm} \cite[Thm.~2.6]{cds}
\label{cds3}

Let $(S, C) \in \mathcal K \mathcal C_g$ with $g \ge 11$ and $\Cliff(C) \ge 3$. Then
$$\dim \Ker d{c_g}_{(S,C)} = \dim c_g^{-1}(C) = {\rm cork} \Phi_{\omega_C} -1.$$
\end{thm}

A similar theorem \cite[Thm.~2.19]{cds} holds for the analogous map for Fano threefolds and K3 surfaces. See also \cite{cd1, cd2} for the map $c_g$ in the non-primitive case.

Another interesting case on which one can compute the corank of the Wahl map is for curves $C$ lying on a Gorenstein surface $S$ with $h^1(\O_S)=0$. As a matter of fact, it happens often that ${\rm cork} \Phi_{\omega_C} = h^0(-K_S)$ (see \cite{w6}). If, in addition, $K_S+C$ is ample and $C$ is not one of a special list (see loc. cit.), then Wahl gives in \cite[Thm. 4.5]{w6} an explicit construction of the universal extension of $C$ in the canonical embedding. In particular all extensions of $C$ have bad singularities. See Remark \ref{cansha} for the case of plane curves.
 

We end this section by recalling what happens, and what is known, for curves $C$ with $\Cliff(C) \le 2$.

The Wahl map, in these cases, was studied, among others, by Ciliberto-Miranda, Wahl and Brawner. We collect here what is known (see \cite{br1, br3} for more details).

\begin{thm} \cite{cm, w1, w3, br1, br2, br3}
\label{altre}

Let $C$ be a smooth irreducible curve of genus $g$. Then:
\begin{itemize}
\item[(i)] If $C$ is hyperelliptic of genus $g \ge 2$, then ${\rm cork} \Phi_{\omega_C}=3g-2$ (in fact this characterizes hyperelliptic curves);
\item[(ii)] If $C$ is trigonal and $g \ge 4$, then ${\rm cork} \Phi_{\omega_C}=g+5$;
\item[(iii)] If $C$ is a plane curve of degree at least $5$, then ${\rm cork} \Phi_{\omega_C}=10$;
\item[(iv)] If $C$ is a general tetragonal curve and $g \ge 7$, then ${\rm cork} \Phi_{\omega_C}=9$.
\end{itemize}
\end{thm}

\begin{remark}
\label{cansha}
In \cite[Ex.~9.7]{cds} (see also \cite{w6}), a $10$-extension of the canonical embedding of any smooth plane curve of degree $d \ge 4$ is constructed. Together with Zak-L'vovsky's theorem, this shows that, if $d \ge 7$, then the curve is precisely (${\rm cork} \Phi_{\omega_C}$)-extendable. On the other hand, if $d=5, 6$, then Zak-L'vovsky's theorem does not apply since $H^0(N_{C/\PP^{g-1}}(-2))\ne 0$ (see \cite{cm}) and $\alpha(C) = {\rm cork} \Phi_{\omega_C}=10 > g-1$. 
As it turns out, these curves actually are $(\alpha(C)+1)$-extendable (see \cite[Rmk.~3.7]{cd2} and the appendix).
 
As far as we know, it is not known if trigonal or tetragonal curves are precisely (${\rm cork} \Phi_{\omega_C}$)-extendable.
\end{remark}

\section{Extendability of Enriques surfaces and Enriques-Fano threefolds}

For Enriques surfaces we also have that they are not extendable if the sectional genus is large enough. 
It is an application of Theorem \ref{anysurface}.

\begin{thm} \cite{klm}

Let $S \subset \PP^r$ be an Enriques surface of sectional genus $g \ge 18$. Then $S$ is not extendable.
\end{thm}

A more precise result for $g = 15$ and $17$ is proved in \cite[Prop.~12.1]{klm}. Actually a complete list of line bundles associated to possibly extendable very ample linear systems on an Enriques surface is available to the authors. See also below for more precise results.
 
As for the extendability of Enriques surfaces, let us define

\begin{defi}
An {\it Enriques-Fano threefold} is an irreducible three-dimensional variety $X \subset \PP^N$ having a hyperplane section $S$ that is a smooth Enriques surface, and such that $X$ is not a cone over $S$. We will say that $X$ has genus $g$ if $g$ is the genus of its general curve section.
\end{defi}

In analogy with Fano threefolds, we obtain a genus bound. The bound was also proved, with completely different methods, by Prokhorov \cite{p2} (see also \cite{k, k3}). Moreover he produced an example of genus $17$. For possible genera, see also \cite{b, sa} for cyclic quotient terminal singularities and \cite[Prop.~13.1]{klm},  \cite{p2} for examples with other singularities. A comprehensive list of all known Enriques-Fano threefolds and their properties can be found in Martello's thesis \cite{ma}.

\begin{thm} \cite[Thm.~1.5]{klm}

Any Enriques-Fano threefold has genus $g \le 17$.
\end{thm}

Both the study of extendability of Enriques surfaces and the moduli map analogous to the one on K3 surfaces (see \eqref{pi}), were recently vastly extended in \cite{cdgk}. We will recall only some of the many results contained in that paper, referring the reader to \cite{cdgk} for a more detailed version.

\begin{defi}
Let $S$ be an Enriques surface and let $L$ be a line bundle on $S$ such that $L^2>0$. We set
$$\phi(L) = \min\{E \cdot L : E\in NS(S), E^2 =0, E>0\}.$$
Let $\E_{g,\phi}$ be the moduli space of pairs $(S, L)$ where $S$ is an Enriques surface, $L$ is an ample line bundle with $L^2=2g-2$ and $\phi(L)=\phi$. Let $\E\mathcal C_{g,\phi}$ be the moduli space of triples $(S, L, C)$ where $(S,L) \in \E_{g,\phi}$ and $C \in |L|$ is a smooth irreducible curve. Let $\mathcal R_g$ be the moduli space of Prym curves, that is of pairs $(C, \eta)$ with $C$ a smooth irreducible curve of genus $g$ and $\eta$ a non-zero $2$-torsion element of $\Pic^0(C)$. 
\end{defi}
Note that $\E_{g,\phi}$ has in general many irreducible components. There is a classification in low genus in \cite[\S 2]{cdgk} and more generally in \cite{kn}.

We have the diagram
$$\xymatrix{\E\mathcal C_{g,\phi} \ar[dr]_{c_{g,\phi}} \ar[r]^{\chi_{g,\phi}} & \mathcal R_g \ar[d]^{f_g} \\ & \mathcal M_g}$$
where $\chi_{g,\phi}(S, L, C) = (C, {K_S}_{|C})$, $f_g$ is the degree $2^{2g} - 1$ forgetful covering map and $c_{g,\phi}=f_g \circ \chi_{g,\phi}$.

Then we have

\begin{thm} \cite[Thm.~1, 2, 3]{cdgk}

\begin{itemize}
\item[(i)] If $\phi \ge 3$, then $\chi_{g,\phi}$ is generically injective on any irreducible component of $\E\mathcal C_{g,\phi}$, with the exception of 10 components where the dimension of the general fiber is given in the list in \cite[Thm.~1]{cdgk};
\item[(ii)] $\chi_{g,2}$ is generically finite on all irreducible components of $\E\mathcal C_{g,2}$ when $g \ge 10$. For $g \le 9$ the dimension of a general fiber of $\chi_{g,2}$ on the various irreducible components of $\E\mathcal C_{g,2}$ is given in the list in \cite[Thm.~2]{cdgk};
\item[(iii)] The dimension of a general fiber of $\chi_{g,1}$ and of $c_{g,1}$ is $\max\{10-g, 0\}$. Hence $c_{g,1}$ dominates the hyperelliptic locus if $g \le 10$ and is generically finite if $g \ge 10$.
\end{itemize}
\end{thm}

This has the following nice consequence

\begin{cor} \cite[Cor.~1.1]{cdgk}

A general curve of genus $2, 3, 4$ and $6$ lies on an Enriques surface, whereas a general curve of genus $5$ or $\ge 7$ does not. A general hyperelliptic curve of genus $g$ lies on an Enriques surface if and only if $g \le 10$.
\end{cor}

The authors then went on to study the extendability of Enriques surfaces (related to the fiber dimension above), proving

\begin{thm} \cite[Cor.~1.2]{cdgk}

Let $S \subset \PP^r$ be an Enriques surface not containing any smooth rational curve. If $S$ is $1$-extendable, then $(S, \O_S(1))$ belongs to the following list (see \cite{cdgk} for precise definitions)
$$\E_{17,4}^{(IV)^+}, \E_{13,4}^{(II)^+}, \E_{13,3}^{(II)}, \E_{10,3}^{(II)}, \E_{9,4}^+, \E_{9,3}^{(II)}, \E_{7,3}.$$
Furthermore, the members of this list are all at most $1$-extendable, except for members
of $\E_{10,3}^{(II)}$, which are at most 2–extendable, and of $\E_{9,4}^+$, which are at most $3$-extendable.
\end{thm}

See also \cite[Rmk.~6.1]{cdgk} for nodal Enriques surfaces.

\section{Extendability of curves in other embeddings}
 
We will collect a sample of non-extendability results for curves. We will not treat, in this section, the case of canonical curves, as it was the object of a specific section.

Perhaps the first significant one is the non-extendability of elliptic normal curves of degree at least $10$, as followed by Del Pezzo \cite{dp} and Nagata's \cite{na} classification of surfaces of degree $d$ in $\PP^d$. Another important result is Castelnuovo-Enriques's theorem \cite{ca, e}, that proves that for a curve of degree large enough with respect to the genus, the only smooth extensions can be ruled surfaces (see \cite{cr, ha} for a modern proof; see also \cite{hm}). 

There are of course many other results in this direction. Our choice is to mention the ones that are relevant to the approach via Zak-L'vovsky's theorem and Gaussian maps.

One important distinction to be made is between the case of {\it any} curve, and the case of a {\it general} curve, that is, a curve with general moduli.

For a given curve we have the following results, according to the Clifford index of $C$. 

\begin{thm} \cite[Cor.~2.10]{kl}, \cite[Thm.~2]{bel}

Let $C \subset \PP^r$ be a smooth irreducible nondegenerate linearly normal curve of genus $g \ge 4$ and degree $d$. Then $C$ is not extendable if:
\begin{itemize}
\item [(i)] $C$ is trigonal, $g \ge 5$ and $d \ge \max\{4g-6,3g+7\}$;
\item [(ii)] $C$ is a plane quintic and $d \ge 26$;
\item [(iii)] $\Cliff(C) = 2$ and $d \ge 4g-3$;
\item [(iv)] $\Cliff(C) \ge 3$ and $d \ge 4g+1-3\Cliff(C)$.
\end{itemize}
\end{thm}

For curves with general moduli we have

\begin{thm} \cite[Cor.~1.7]{lo}

Let $C \subset \PP^r$ be a curve of genus $g$ with general moduli and degree $d$. Then $C$ is not extendable if:
\begin{itemize}
\item [(i)] $d \ge 2g+15$ for $3 \le g \le 4$;
\item [(ii)] $d \ge 2g+13$ for $5 \le g \le 8$;
\item [(iii)] $d \ge 2g+10$ for $g \ge 9$.
\end{itemize}
\end{thm}

For general embeddings of curves with general moduli we have

\begin{thm} \cite[Cor.~1.7]{lo}

Let $C \subset \PP^r$ be a general linearly normal degree $d$ embedding a curve of genus $g$ with general moduli. Then $C$ is not extendable if:
\begin{itemize}
\item [(i)] $d \ge g+14$ for $3 \le g \le 4$;
\item [(ii)] $d \ge g+12$ for $5 \le g \le 8$;
\item [(iii)] $d \ge g+9$ for $g \ge 9$.
\end{itemize}
\end{thm}

Another possibility is to study the extendability of curves in terms of general Brill-Noether embeddings.

\begin{thm} \cite[Thm.~1.9]{lo}

Let $d, g, r$ be integers such that $\rho(d,g,r) \ge 0, d<g+r, r \ge 11$ or $r = 9$. If $C$ is a curve with general moduli and $L$ is a general line bundle in $W^r_d(C)$, then $C \subset \PP^r = \PP H^0(L)$ is not extendable.
\end{thm}

\section{Extendability of other surfaces}
\label{extsu}

In this section we will outline some results about the extendability of surfaces that are neither $K3$ nor Enriques surfaces.

Most results that we will mention are applications of the following

\begin{thm} \cite[Thm.~1.1]{klm}
\label{anysurface}

Let $S \subset \PP^r$ be a smooth irreducible linearly normal surface and let $H$ be its
hyperplane bundle. Assume there is a base-point free and big line bundle $D_0$ on $S$
with $H^1(H-D_0)=0$ and such that the general element $D \in |D_0|$ is not rational
and satisfies:
\begin{itemize}
\item[(i)] the Gaussian maps $\Phi_{H_D, \omega_{D}(D_0)}$ is surjective;
\item[(ii)] the multiplication maps $\mu_{V_D, \omega_D}$ and $\mu_{V_D, \omega_{D}(D_0)}$ are
surjective 
\end{itemize}
where
$V_D := \Im \{H^0(S, H-D_0) \to H^0(D, (H-D_0)_{|D})\}$.
Then
\[ \alpha(S) \le {\rm cork} \Phi_{H_D, \omega_D}. \]
\end{thm}

Now define

\begin{defi}
\label{m(L)}
Let $S$ be a smooth surface and let $L$ be an effective line bundle on $S$ such that the general
divisor $D \in |L|$ is smooth and irreducible. We say that $L$ is {\it trigonal}, if $D$ is such. We denote by $\Cliff(L)$ the Clifford index of $D$. Moreover, when $L^2 > 0$, we set
\[ \mbox{$\varepsilon(L) = 3$ if $L$ is trigonal; $\varepsilon(L) =5$ if $\Cliff(L) \geq 3$, $\varepsilon(L) =0$ if $\Cliff(L) = 2$;} \]
\[ m(L) =  \begin{cases} \frac{16}{L^2} & {\rm if} \  L.(L+K_Y) = 4; \\ \frac{25}{L^2} &
{\rm if} \ L.(L+K_Y) = 10 \ {\rm and \ the \ general } \\ 
& {\rm divisor \ in} \ |L| \ {\rm is \ a \ plane \
quintic}; \\ \frac{3L.K_Y + 18}{2L^2} + \frac{3}{2} & {\rm if} \ 6 \leq L.(L+K_Y) \leq 22 \
{\rm and} \ L \ {\rm is \ trigonal}; \\ \frac{2 L.K_Y - \varepsilon(L)}{L^2} + 2 & {\rm otherwise}.
\end{cases} \]
\end{defi}

We give first a general result, application of Theorem \ref{anysurface}, that holds on any surface, once given a suitable embedding. 

\begin{thm} \cite[Cor.~3.3]{klm}

Let $S \subset \PP V$ be a smooth irreducible surface with $V \subseteq H^0(mL+D)$, where $L$ is a base-point free, big, nonhyperelliptic line bundle on $S$ with $L \cdot (L+K_S) \ge 4$ and $D\ge 0$ is a divisor. Suppose that $S$ is regular or linearly normal and that $m$ is such that $H^1((m-2)L + D) = 0$ and $m > \max\{m(L) - (L\cdot D)/L^2, (L \cdot K_S +2-L\cdot D)/L^2 + 1\}$. Then $S$ is not extendable.
\end{thm}

See also \cite[Cor.~3.5]{klm} for a better result in the case of adjoint embeddings.

\vskip .3cm

We now list what is known in terms of the Kodaira dimension. We will concentrate on minimal surfaces.

\subsection{Surfaces of Kodaira dimension $-\infty$}

\hskip 7cm

For embeddings of $\PP^2$ the non-extendability was settled by Scorza \cite{s1}. See also Proposition \ref{h1t}. 

As explained in \cite[\S 5.5]{bs}, if a variety $X$ is a $\PP^d$-bundle over a smooth variety and $X$ is an ample divisor on a locally complete intersection variety $Y$, then in almost all cases, including the case of surfaces, $Y$ is a $\PP^{d+1}$-bundle over the same base and $X$ belongs to the tautological linear system. This shows that l.c.i. extensions of  $\PP^d$-bundles are $\PP^{d+1}$-bundles. 

On the other hand, as far as we know, not so many results are known, in general, about the extendability of ruled surfaces. 

A result was proved in \cite[Thm. 3.3.22]{db} for rational ruled surfaces, as an application of Theorem \ref{anysurface}. We believe that this can be extended to other ruled surfaces.

\subsection{Surfaces of Kodaira dimension $0$}

\hskip 7cm

The cases of K3 and Enriques surfaces have been treated above. Abelian and bi-elliptic surfaces are not extendable by Proposition \ref{h1t}.

\subsection{Surfaces of Kodaira dimension $1$}

\hskip 7cm

The only results that we are aware of are the ones in \cite{lms}, which hold especially for Weierstrass fibrations.

\subsection{Surfaces of general type}

\hskip 7cm

We are not aware of results, except the one below, that applies, for example, for pluricanonical embeddings.

\begin{thm} \cite[Cor.~1.2]{klm}

Let $S \subset \PP V$ be a minimal surface of general type with base-point free and nonhyperelliptic canonical bundle and $V \subset H^0(mK_S + D))$, where $D \ge 0$ and either $D$ is nef or $D$ is reduced and $K_S$ is ample. Suppose that $S$ is regular or linearly normal and that 
\[ m \geq \begin{cases} 9 & {\rm if} \ K_Y^2 = 2; \\ 7 & {\rm if} \ K_Y^2 = 3; \\ 6 & {\rm if}
\ K_Y^2 = 4 \ {\rm and \ the \ general \ curve \ in} \ |K_Y| \ {\rm is \ trigonal \ or \ if }  \\ & K_Y^2 = 5 {\rm \ and \ the \ general \ curve \ in} \ |K_Y| \ {\rm is \ a \ plane \ quintic;}
\\ 5 & {\rm if \ either \ the \ general \ curve \ in} \ |K_Y| \ {\rm has \ Clifford \ index} \ 2 \
{\rm or} \\ & 5 \leq K_Y^2 \leq 9 \ {\rm and \ the \ general \ curve \ in} \ |K_Y| \ {\rm is \
trigonal;} \\ 4 & {\rm otherwise}.
\end{cases} \]
Then $S$ is not extendable.
\end{thm}

\appendix
\section{Extendability of canonical models of plane quintics} 
\centerline{\scalebox{.8}{by THOMAS DEDIEU}}

\bigskip

\def\P{\mathbf{P}}
\def\C{\mathbf{C}}
\newcommand{\bx}{\mathbf x}
\def\epsilon{\varepsilon}
\newcommand{\trsp}[1]{\vphantom{#1}^{\mathsf T\!} #1}
\newcommand{\ie}{i.e.,\ } 

\newcounter{enonce}
\def\theenonce{\thethm.\arabic{enonce}}
\def\sousprop{\newline\refstepcounter{enonce}{\normalfont\theenonce}.}

\renewcommand{\theequation}{\thethm.\arabic{equation}}

Let $C$ be a smooth plane curve of degree $d \geq 5$, which we most
often consider in its canonical embedding in $\P^{g-1}$, $g=\tfrac 12
(d-1)(d-2)$. 
By \cite[Rmk.~4.9]{wahl90} one has $\alpha(C)=10$. 
If $d \geq 7$ then $C$ has Clifford index strictly larger than $2$,
hence $h^0(N_{C/\P^{g-1}}(-k))=0$ for all $k>1$,
and Theorems \ref{zak} and \ref{cds1} apply together, to the
effect that $C$ is precisely $10$-extendable; a universal extension of
$C$ is described in \cite[\S 9]{cds-thomas}, which shows that all surface
extensions of $C$ are rational.
If $d=5$ (resp.\ $6$), then $C$ has Clifford index $1$ (resp.\ $2$), 
$h^0(N_{C/\P^{g-1}}(-2))=3$ (resp.\ $1$), and
$h^0(N_{C/\P^{g-1}}(-k))=0$ for all $k>2$, see
\cite{knutsen18} and the references therein.
The expectation in this case is that there should exist a $12$ (resp.\ $10$)
dimensional family of non-isomorphic surface extensions of $C$, and
inside this a $2$ (resp.\ $0$) dimensional family of surfaces
in which the first infinitesimal neighbourhood of $C$ is trivial,
see \cite[\S 1]{cd-double}.
The case $d=6$ is studied in detail in \cite{cd-double} and
\cite[\S 3.2]{cd-higher}: the above expectations are met, and a
suitable embedding of the
double cover of $\P^2$ branched over $C$ is a $K3$ surface
extending $C$; this $K3$ surface is an anticanonical divisor of the
weighted projective space $\P(1^3,3)$, and so are all $K3$ extensions
of $C$. In this appendix I show
that a similar situation holds in the case $d=5$.

I am grateful to the anonymous referee for pointing out
\cite[Thm.~7.2]{saint-donat}, which triggered a substantial
improvement of the results presented in this text.

\begin{prop}
\label{th:main}
Let $C$ be a smooth plane quintic, in its canonical embedding $C
\subset \P^5$.
\sousprop{}\label{th:K3}
The moduli space of $K3$ extensions of $C$ has dimension $12$.
\sousprop{}\label{th:K3triv}
Inside the moduli space of $K3$ extensions of $C$,
there is a $2$-dimensional family of extensions
in which the first infinitesimal neighbourhood of $C$ is
trivial.
\sousprop{}\label{th:Totaro}
There exists a $14$-extension of $C$, which is an arithmetically
Gorenstein normal Fano variety of dimension $15$ and index $13$.
\end{prop}

The precise meaning of \ref{th:K3} is that
the fibre of the moduli map
$c_6: \mathcal {KC}_6 \to \mathcal {M}_6$ (introduced in \eqref{pi})
over a general plane quintic has dimension $12$.
By way of comparison, note that $c_6$ is surjective and
has general fibre of dimension $10$.

Our story begins with
the following construction of Ide \cite[\S 4.2]{ide}.
\medskip
\refstepcounter{thm}
\paragraph{\bfseries Construction \thethm}
\label{th:construction}
Let $f$ be a degree $5$ homogeneous polynomial in $\bx =
(x_0,x_1,x_2)$ defining $C \subset \P^2$, and $\ell$ a linear
functional in $\bx$ defining a line $\Gamma$ intersecting $C$ transversely.
Then the degree $5$ weighted hypersurface $S$ in $Y=\P(1^3,2)$ defined
by 
\[
  \ell(\bx) y^2 + f(\bx) = 0,
\]
in homogeneous coordinates $(\bx,y)$ so that the
$x_i$ have weight $1$ and $y$ has weight $2$,
is a $K3$ surface (with an ADE singularity), and $C$ is the complete
intersection of $S$ and the degree $2$ hypersurface
$\Pi$ defined by $y=0$.

A geometric construction of $S$ is as follows. Let
$\epsilon: P' \to \P^2$ be the
blow-up at the five points of $\Gamma\cap C$, and $\Gamma'$ and $C'$
be the proper transforms of $\Gamma$ and $C$ respectively. Let
$\pi:S' \to P'$ be
the double cover branched over the smooth curve $\Gamma'+C'$. 
The pull-back $\pi^*\Gamma'$ is a $(-2)$-curve, and $S$ is obtained
from $S'$ by contracting it to an $A_1$ double point.

The latter contraction is realized by the map given by the complete
linear system $|-\pi^* K_{P'}|$.
Arguing as in \cite[Prop.~2.2]{cd-double}, one sees that the latter
has dimension $6$, 
has $\pi^*\Gamma'$ as a fixed part, and the general member of its moving
part is mapped birationally by $\epsilon\circ\pi$ to a smooth quintic
$C_1$ passing through the $5$ points of $\Gamma \cap C$ and otherwise
everywhere tangent to $C$.
Note that $|-K_{P'}|$ has $\Gamma'$ as fixed part and the pull-back by
$\epsilon$ of the linear system of conics as moving part, hence
for all $D \in |C'|$ its restriction to $D$ is the canonical divisor $K_D$.
This implies that the image of the map given by $|-\pi^* K_{P'}|$ is a
$K3$ surface $S \subset \P^6$ (with one ordinary double point), having
$C$ as a hyperplane section. 
As a sideremark, note that $|-\pi^* K_{P'}|$ maps the pull-backs by $\pi$
of the $5$ exceptional curves of $\epsilon$ to $5$ lines passing through
the node of $S$.

The above description is conveniently complemented by adopting the following
equation based point of view.

\bigskip
\refstepcounter{thm}
\paragraph{\bfseries Weighted projective geometry \thethm}
\label{p:weighted}
For useful background on the subject, I recommend
\cite[Exercises V.1.3]{kollar} and \cite{iano-fletcher}.
The weighted projective space $Y=\P(1^3,2)$ has only one singular point, namely
the coordinate point 
$(0:0:0:1)$ at which it has a quotient singularity of type
$\frac 12 (1,1,1)$, that is an ordinary quadruple point (locally
isomorphic to the cone over the Veronese surface $v_2(\P^2)$).
It has dualizing sheaf $\O(5)$, which is not invertible. All
(weighted) quintic hypersurfaces in $Y$ pass through the singular
point $(0:0:0:1)$, and the general such has an ordinary double point
there.

By adjunction, the general quintic surface in $Y$ has trivial
(invertible) canonical sheaf, hence is a $K3$ surface (with canonical
singularities). Moreover, the general complete intersection of
hypersurfaces of degrees $5$ and $2$ is smooth, with
(invertible) canonical sheaf $\O(2)$.

Maintaining the notation of \ref{th:construction}, the curve $C$ may
be embedded in $Y$ as the complete intersection defined by
\begin{equation}
\label{th:CI}
  y = f(\bx)=0.
\end{equation}
On the other hand, the automorphisms of $Y$ are given in homogeneous
coordinates by
\begin{equation}
\label{th:aut}
  (\bx:y) \mapsto (A \bx: a y+Q(\bx))
\end{equation}
with $A \in \GL(3)$, $a \in \C^*$, and
$Q \in H^0(\P(1^3),\O(2))$,
see for instance \cite[\S 8]{alamrani}.
One may therefore put every complete intersection of hypersurfaces
of degrees $2$ and $5$ in the form \eqref{th:CI}
by acting with an automorphism of $Y$, which shows that it is isomorphic
to a plane quintic curve. 

The sheaf $\O(2)$ is invertible on $Y$, and the associated complete
linear system induces an embedding $\phi: Y \to \P^6$ with image a
cone over a Veronese surface 
$v_2(\P^2) \subset \P^5$, with vertex a point.
The map $\phi$ sends $S$ to a $K3$ surface of degree $10$ in $\P^6$
passing through the vertex, and having
$C$ (in its canonical embedding) as a hyperplane section: in other
words $\phi(S)$ is an extension of the canonical curve $C$.
It coincides with the model of $S$ in $\P^6$ given by
$|-\pi^*K_{P'}|$ and described in \ref{th:construction}.
The sheaf $\O_Y(1)$ is not invertible on $Y$ and neither is its
restriction to $S$; the associated complete linear series induces
the rational map $S \dashrightarrow \P^2$ coinciding with
$\epsilon\circ\pi$
off the node.

\begin{proof}[Proof of Proposition~\ref{th:main}]
As we have seen, we may consider $C$ as the complete intersection in
$Y$ of two hypersurfaces of degrees $2$ and $5$ as in \eqref{th:CI}.
The linear system of quintics containing $C$ has
dimension
\[
  h^0(Y, {\mathcal I}_{C/Y}(5)) - 1 = h^0(Y, \O_Y(3)) = 13,
\]
and its general element gives a $K3$ surface extending $C \subset
\P^5$.
By the description in \eqref{th:aut},
the automorphisms of $Y$ fixing the hypersurface
$y=0$ are all of the form
$(\bx:y) \mapsto (\bx: a y)$,
hence they form a $1$-dimensional group.
From this we conclude that there
is a $12$-dimensional family of mutually non-isomorphic $K3$
extensions of $C$.

On the other hand, by \cite[Rmk.~7.13]{saint-donat} every $K3$ surface
containing $C$ is contained in a cone over the Veronese surface, that
is in $\P(1^3,2)$ in its model given by $\O(2)$.
By the computation of the divisor class group of weighted projective
spaces \cite[Thm.~7.1]{beltrametti-robbiano} and adjunction, such a surface is
necessarily of the same kind as those considered above, \ie it is a
weighted quintic in $\P(1^3,2)$ in which $C$ is cut out by a quadric.
Therefore the $12$-dimensional family found above is the full fibre of
$c_6$ over $C$, and \ref{th:K3} is proved.

Inside this family, there is a $2$-dimensional family of surfaces in which
the first infinitesimal neighbourhood of $C$ is trivial, namely
those surfaces constructed as in \ref{th:construction} by taking a double
cover branched over the disjoint union of $C$ itself and a line:
that the infinitesimal neighbourhood is indeed trivial in this case
follows from the argument
given in \cite[p.~875]{beauville-merindol}.
We get a $2$-dimensional family by letting the line $\Gamma$ move
freely in $\P^2$:
the double cover of the blow-up remembers the $5$ points of
$\Gamma\cap C$, hence
two general lines give two non-isomorphic surfaces. 
This proves \ref{th:K3triv}.

Eventually, to prove \ref{th:Totaro}
we use Totaro's
construction as in \cite[\S 3.2]{cd-higher}.
Namely we take $X$ the quintic hypersurface
in $\P(1^3,2,2^{13})$ given in homogeneous coordinates
$(\bx:y:\mathbf z)$
by
\[
  f(\bx) + G_1(\bx,y)z_1+ \cdots +
  G_{13}(\bx,y)z_{13}=0,
\]
where $G_1,\ldots,G_{13}$
form a basis of $H^0(Y,\O(3))$,
and embed it in $\P^{19}$ with the complete linear system $|\O_X(2)|$,
which is $\tfrac 1 {13}$ of the anticanonical series.
That $X\subset \P^{19}$ is arithmetically Gorenstein and normal
follows from the fact 
that it has canonical curves as linear sections as in \cite[\S 5]{cds-thomas}.
\end{proof}

\begin{stremark}\label{rk:genext}
Not all the $K3$ extensions of a plane quintic $C$ may be constructed
as in \ref{th:construction}, as a dimension count shows.
In fact, by \cite[Thm.~7.2]{saint-donat} if $S'$ is a smooth $K3$ surface
containing $C$, then $C \sim 2L+\Gamma'$ in $S'$, where $L$ and $\Gamma'$ are
irreducible curves of respective genera $2$ and $0$,
and $L \cdot \Gamma' = 1$.
The linear system $|L|$ induces a double cover $\pi:S' \to \P^2$
branched over a sextic $B \subset \P^2$ mapping $\Gamma'$ to a line
$\Gamma$, and there are the following two possibilities.

a) If $\Gamma'$ is contained in the ramification locus of $\pi$, then
$B$ contains $\Gamma$, and in this case $S'$ and its image $S$ by
$|C|$ are as in the Construction~\ref{th:construction}.

b) If $\Gamma'$ is not contained in the ramification locus of $\pi$,
then $\Gamma$ is a tri-tangent line to $B$ and $\pi^* \Gamma =
\Gamma'+\Gamma''$, with $\Gamma''$ another irreducible rational curve
such that $L \cdot \Gamma'' = 1$, and $\Gamma' \cdot \Gamma'' = 3$.
Then $L \sim \Gamma' + \Gamma''$ as it is linearly equivalent to the
pull-back of a line by $\pi$.
One thus has
\[
  C \sim 2L+\Gamma'
  \sim 3L-\Gamma'',
\]
and the general member of $|C|$ is mapped birationally by $\pi$ to a
smooth plane quintic everywhere tangent to $B$.
The map induced by $|C|$ contracts $\Gamma'$ and maps $\Gamma''$ to a
degree $5$ curve spanning a $\P^3$ and with a space triple point at
the image of $\Gamma'$.

Type b) extensions of $C$ are thus double covers of the plane
branched over a sextic $B$ which is everywhere tangent to the sextic
$C+\Gamma$ for some line $\Gamma$.
By \cite[Thm.~4.1]{cd-double} there is a $10$-dimensional family of
such $K3$ extensions for each choice of $\Gamma$, hence a
$12$-dimensional family if we let $\Gamma$ move.
The general extension of $C$ is thus of type b).

This may also be seen directly from the equations. Let $C$ be given
in $Y$ by equations $y=f(\bx)=0$ as in \eqref{th:CI}. Then the
general quintic $S$ in $Y$ containing $C$ has equation
\begin{equation}\label{th:geneq}
  \ell(\bx) y^2 + 2 g(\bx) y + f(\bx) = 0
\end{equation}
with $g$ a non-zero degree $3$ homogeneous polynomial, and $\ell$ as
before a non-zero linear form. To see this surface birationally as a
double cover of the plane, we put the equation \eqref{th:geneq} in the
form 
\begin{equation*}
  \ell \Bigl( y + \frac g \ell \Bigr)^2 +
  \Bigl(f- \frac {g^2} \ell \Bigr) = 0,
\end{equation*}
which shows that $S$ is birational to the double cover branched over
the plane sextic $(\ell f -g^2=0)$, which is everywhere tangent to
both $C$ and the line $(\ell = 0)$.
\end{stremark}

\begin{stremark}
As has been observed above in \ref{p:weighted} and \ref{rk:genext},
all $K3$ extensions of $C$ are singular.
This contrasts with the fact, pointed out to me by Edoardo Sernesi, that
the rational surface extensions of canonical plane curves constructed
in \cite[\S 9]{cds-thomas} 
have one elliptic singularity which in the case of quintics is smoothable.
These rational surface extensions may only be partially
smoothed, to $K3$ surfaces with one ordinary double point, 
meaning that there is a global obstruction to a total smoothing.
\end{stremark}

\begin{stremark}
As was the case for plane sextics \cite[Rmk~3.7]{cd-higher}, canonical
plane quintics provide curves which are not complete intersections,
for which the assumptions of either form of Theorem~\ref{zak} don't
hold, and which are extendable more than $\alpha$ times.

Moreover, canonical plane quintics show that Zak's claim
\cite[p.~278]{zak}, that if an $n$-dimensional $X \subset \P^N$ has
$\alpha(X)<(N-n-1)(N+1)$ then it is at most $\alpha(X)$-extendable,
cannot hold without an additional assumption.
\end{stremark}

\begin{stremark}
We can play the same game with plane quartics and cubics as with
sextics and quintics.
Let $\epsilon:P' \to \P^2$ be the blow-up at the $8$ (resp.\ $9$)
points of the transverse intersection
$\Gamma\cap C$, with $\Gamma$ a conic and $C$ a smooth quartic
(resp.\ $E_1\cap E_2$, with $E_1$ and $E_2$ two smooth cubics),
and $\pi: S' \to P'$ be the double cover branched over the proper
transforms of $\Gamma$ and $C$ (resp.\ $E_1$ and $E_2$). 
In the former case, the linear system $|-\pi^*K_{P'}|$ maps $S'$ to a
quartic surface $S$ in $\P(1^3,1)=\P^3$ with
an ordinary double point.
In the latter case we get what I would describe as a virtual cubic
surface in $\P(1^3,0)$; the linear system 
$|-\pi^*K_{P'}|$ is then of the form $|2F|$ with $F$ an
elliptic curve, hence gives a map to $\P^2$ contracting $S'$ onto a
smooth conic.
\end{stremark}

\vskip .3cm

\providecommand{\bysame}{\leavevmode\hbox to3em{\hrulefill}\thinspace}
\providecommand{\og}{``}
\providecommand{\fg}{''}
\providecommand{\smfandname}{and}
\providecommand{\smfedsname}{eds.}
\providecommand{\smfedname}{ed.}
\providecommand{\smfmastersthesisname}{M\'emoire}
\providecommand{\smfphdthesisname}{Th\`ese}

\end{document}